\documentclass[11pt,reqno]{amsart}

\usepackage{amssymb}
\usepackage{amsmath,amsthm,amsfonts,amssymb,latexsym,mathrsfs,color,hyperref}
\usepackage{graphicx}
\usepackage{xspace}
\usepackage{multirow}
\usepackage{diagbox}
\usepackage{capt-of}
\usepackage{tikz}
\tikzset{
	level 1/.style = {sibling distance = 1.5cm},
	level 2/.style = {sibling distance = 0.8cm},
    level distance = 0.9 cm
}
\usetikzlibrary {decorations.pathmorphing}
\usetikzlibrary{matrix, positioning}
\tikzstyle{snakeline} = [decorate, decoration={snake, amplitude=.4mm, segment length=2mm}]

\usepackage{tikz-cd}
\usepackage[boxsize=1.3em, centerboxes]{ytableau}

\usepackage{tikz-qtree}
\tikzset{every tree node/.style={minimum width=0.1cm,draw,circle},
         blank/.style={draw=none},
         edge from parent/.style=
         {draw,edge from parent path={(\tikzparentnode) -- (\tikzchildnode)}},
         level distance=0.8cm}

\newtheorem{theorem}{Theorem}%[section]

\newtheorem{proposition}[theorem]{Proposition}
\newtheorem{conjecture}[theorem]{Conjecture}

\newtheorem{lemma}[theorem]{Lemma}

\newtheorem{example}[theorem]{Example}

\newcommand{\Sta}{{\rm Stab\,}}
\newcommand{\maj}{{\rm maj\,}}

\newcommand{\Stab}{{\rm Stab\,}}

\newcommand{\Des}{{\rm Des\,}}

\newcommand{\des}{{\rm des\,}}

\newcommand{\exc}{{\rm exc\,}}

\newcommand{\fix}{{\rm fix\,}}

\newcommand{\mdn}{\mathcal{D}}

\newcommand{\msn}{\mathfrak{S}_n}

\newcommand{\lrf}[1]{\lfloor #1\rfloor}

\newcommand{\z}{ \mathbb{Z}}

\title{A curious symmetric decomposition of the (des, exc)-Eulerian polynomials}
\author[S.-M.~Ma]{Shi-Mei Ma}
\address{School of Mathematics and Statistics, Shandong University of Technology, Zibo 255049, Shandong, China}
\email{shimeimapapers@163.com (S.-M. Ma)}
\author[T.~Mansour]{Toufik Mansour}
\address{Department of Mathematics, University of Haifa, 3498838 Haifa, Israel}
\email{tmansour@univ.haifa.ac.il(T.~Mansour)}
\author[Y.-N. Yeh]{Yeong-Nan Yeh}
\address{College of Mathematics and Physics, Wenzhou University, Wenzhou 325035, P.R. China}
\email{mayeh@math.sinica.edu.tw (Y.-N. Yeh)}
\subjclass[2010]{Primary 05A05; Secondary 12D05}
\begin{document}

\maketitle
\begin{abstract}
One of the most central result in combinatorics says that the descent statistic and the excedance statistic are equidistribued over the symmetric group.
As a continuation of the work of Shareshian-Wachs (Adv.~Math., 225(6) (2010), 2921--2966),
we provide a curious $t$-symmetric decomposition for
the generating polynomial of the joint distribution of the descent and excedance statistics over the symmetric group.
\bigskip

\noindent{\sl Keywords}: Eulerian polynomials; Descents; Excedances; Symmetric decompositions
\end{abstract}
\date{\today}
%\date{\today}
%%%%%%%%%%%%%%%%%%%%%%%%%%%%%%%%%%%%%%%%%%%%%%%%%%%%%%%%%%%%%%%%%%%%%%%%%%
%%%%%%%%%%%%%%%%%%%%%%%%%%%%%%%%%%%%%%%%%%%%%%%%%%%%%%%%%%%%%%%%%%%%%%%%%%
%%%%%%%%%%%%%%%%%%%%%%%%%%%%%%%%%%%%%%%%%%%%%%%%%%%%%%%%%%%%%%%%%%%%%%%%%%
%%%%%%%%%%%%%%%%%%%%%%%%%%%%%%%%%%%%%%%%%%%%%%%%%%%%%%%%%%%%%%%%%%%%%%%%%%
%%%%%%%%%%%%%%%%%%%%%%%%%%%%%%%%%%%%%%%%%%%%%%%%%%%%%%%%%%%%%%%%%%%%%%%%%%
%%%%%%%%%%%%%%%%%%%%%%%%%%%%%%%%%%%%%%%%%%%%%%%%%%%%%%%%%%%%%%%%%%%%%%%%%%
%%%%%%%%%%%%%%%%%%%%%%%%%%%%%%%%%%%%%%%%%%%%%%%%%%%%%%%%%%%%%%%%%%%%%%%%%%
%%%%%%%%%%%%%%%%%%%%%%%%%%%%%%%%%%%%%%%%%%%%%%%%%%%%%%%%%%%%%%%%%%%%%%%%%%
%%%%%%%%%%%%%%%%%%%%%%%%%%%%%%%%%%%%%%%%%%%%%%%%%%%%%%%%%%%%%%%%%%%%%%%%%%
%%%%%%%%%%%%%%%%%%%%%%%%%%%%%%%%%%%%%%%%%%%%%%%%%%%%%%%%%%%%%%%%%%%%%%%%%%
%%%%%%%%%%%%%%%%%%%%%%%%%%%%%%%%%%%%%%%%%%%%%%%%%%%%%%%%%%%%%%%%%%%%%%%%%%
%%%%%%%%%%%%%%%%%%%%%%%%%%%%%%%%%%%%%%%%%%%
\section{Introduction}
%%%%%%%%%%%%%%%%%%%%%%%%%%%%%%%%%%%%%%%%%%%
%%%%%%%%%%%%%%%%%%%%%%%%%%%%%%%%%%%%%%%%%%%%%%%%%%%%%%%%%%%%%%%%%%%%%%%%%%%%%%%%%
%%%%%%%%%%%%%%%%%%%%%%%%%%%%%%%%%%%%%%%%%%%%%%%
%%%%%%%%%%%%%%%%%%%%%%%%%%%%%%%%%%%%%%%%%%%%%%%%%
In 1736, while working on sums of the form $\sum_{k=1}^mk^nj^k$, Euler discovered a family of polynomials defined by
\begin{equation*}\label{Anxsum}
\sum_{k=0}^\infty k^nx^k=\frac{xA_n(x)}{(1-x)^{n+1}}
\end{equation*}
for $n\geqslant 1$.
The polynomials $A_n(x)$ are now known as {\it Eulerian polynomials}.
Over the past century, they have been
widely studied because of their natural occurrence in many different contexts, such as mathematical analysis~\cite{Liu07,Petersen15}, probability distribution~\cite{Hwang20}, numerical analysis and spline approximation~\cite{He12}, combinatorics and geometry~\cite{Athanasiadis17,Foata08,Petersen15}.

The symmetric group $\msn$ is the set of all permutations of $[n]:=\{1,2,\ldots,n\}$.
For any permutation $\pi=\pi(1)\cdots\pi(n)\in\msn$, we say that $i$ is a {\it descent} (resp.~{\it excedance}, {\it fixed point})
if $\pi(i)>\pi(i+1)$ (resp.~$\pi(i)>i$,~$\pi(i)=i$).
Let $\des(\pi)$, $\exc(\pi)$ and $\fix(\pi)$ be the numbers of descents, excedances and fixed points of $\pi$, respectively.
MacMahon~\cite[Vol.~I, p.~186]{MacMahon15} observed that
$$A_n(x)=\sum_{\pi\in\msn}x^{\des(\pi)}=\sum_{\pi\in\msn}x^{\exc(\pi)},$$
which becomes one of the most central result in combinatorics.

Let $\mdn_n=\{\pi\in\msn: \fix(\pi)=0\}$ be the set of all derangements of $[n]$.
The {\it derangement polynomials} are defined by
$$d_n(x)=\sum_{\pi\in\mdn_n}x^{\exc(\pi)}.$$
The derangement polynomials share several properties with the Eulerian polynomials, including unimodality~\cite{Brenti90}, real-rootedness~\cite{Liu07},
$\gamma$-positivity~\cite{Ma19,Zeng16} and asymptotic normality~\cite{Hwang20}.

The {\it major index} of $\pi\in\msn$ is the defined by
$$\maj(\pi)=\sum_{\pi(i)>\pi(i+1)}i.$$
For $\Omega\subseteq \z$,
let $\Sta(\Omega)$ be the set of all subsets of which do not contain two consecutive integers.
We now recall a classical result.
\begin{theorem}[{\cite[Theorem~2.15]{Athanasiadis17}},{\cite[Remark~5.5]{Shareshian10}}]\label{thm01}
For any $n\geqslant 2$, one has
$$\sum_{\pi\in\mdn_n}q^{\maj(\pi)-\exc(\pi)}p^{\des(\pi)}t^{\exc(\pi)}=\sum_{i=1}^{\lrf{n/2}}\xi_{n,i}(p,q)t^i(1+t)^{n-2i},$$
where $$\xi_{n,i}(p,q)=p\sum_{\pi}p^{\des(\pi^{-1})}q^{\maj(\pi^{-1})},$$
and the sum runs through $\pi\in\msn$ for which $\Des(\pi)\in \Stab([2,n-2])$ has $i-1$ elements.
\end{theorem}

In~\cite{Shareshian10}, Shareshian and Wachs studied the trivariate Eulerian polynomials
\begin{equation*}\label{Wachs}
A_n^{\maj,\des,\exc}(q,p,q^{-1}t)=\sum_{\pi\in\msn}q^{\maj(\pi)-\exc(\pi)}p^{\des(\pi)}t^{\exc(\pi)}.
\end{equation*}
They noted that these polynomials are not $t$-symmetric, see~\cite[p.~2951]{Shareshian10} for details.
For example,
\begin{align*}
A_4^{\maj,\des,\exc}(q,p,q^{-1}t)&=1+(3p+2pq+pq^2+2p^2q^2+2p^2q^3+p^2q^4)t+\\
&(3p+pq+p^2q+3p^2q^2+2p^2q^3+p^3q^4)t^2+pt^4.
\end{align*}

A special case of Theorem~\ref{thm01} is the $q=1$ case, which says that the $(\des, \exc)$-derangement polynomials is $\gamma$-positive,
and so $t$-symmetric. Motivated by this special case, it is natural to consider the $(\des, \exc)$-Eulerian polynomials:
$$A_n(s,t)=\sum_{\pi\in\msn}s^{\des(\pi)}t^{\exc(\pi)}.$$
It is clear that $A_n(1,t)=A_n(t,1)=A_n(t)$.
Let
\begin{equation*}\label{A_n(s,t)gf}
\sum_{n\geqslant 0}A_n(s,t)\frac{u^n}{(1-s)^{n+1}}=\sum_{n\geqslant 0}\sum_{k\geqslant 0}\sum_{r\geqslant 0}f(n,k,r)s^rt^k.
\end{equation*}
In~\cite{Li12}, Li found that
\begin{equation}
f(n,k,r)=[x^{kr}]\frac{(1-x^k)(1-x^{k+1})^n}{(1-x)^{n+1}x^k},
\end{equation}
and showed that the $h*$-polynomial of a half-open hypersimplex is given by
$$\sum_{\substack{\pi\in\msn\\ \exc(\pi)=k}}s^{\des(\pi)}.$$
Moreover, she found that $\#\{\pi\in\msn \mid  \des(\pi)=1, \exc(\pi)=k\}=\binom{n}{k+1}$ for $1\leqslant k\leqslant n-1$.
It should be noted that
the number of descents and excedances will depend on where the fixed points of a permutation are, not just on the number of fixed points.

For $1\leqslant n\leqslant 5$, we stumble upon the following $t$-symmetric decompositions:
\begin{align*}
A_1(s,t)&=1,~A_2(s,t)=1+st=1+t+(s-1)t,\\
A_3(s,t)&=1+(3s+s^2)t+st^2={(1+(1+s)^2t+t^2)}+(s-1)t(1+t),\\
A_4(s,t)&=1+(6s+5s^2)t+(4s+6s^2+s^3)t^2+st^3\\
=&(1+t){(1+5s(1+s)t+t^2)}+(s-1)t{(1+(1+s)^2t+t^2)},\\
A_5(s,t)&=1+(10s+15s^2+s^3)t+(10s+36s^2+19s^3+s^4)t^2+\\
&(5s+15s^2+6s^3)t^3+st^4\\
=&(1+(1+9s+15s^2+s^3)t+(1+14s+36s^2+14s^3+s^4)t^2+\\&(1+9s+15s^2+s^3)t^3+t^4)+(s-1)t(1+t)(1+5s(1+s)t+t^2).
%&A_6(s,t)=1+(15s+35s^2+7s^3)t+(20s+127s^2+133s^3+22s^4)t^2+\\&(15s+111s^2+141s^3+34s^4+s^5)t^3+(6s+29s^2+21s^3+s^4)t^4+st^5\\
%&=(1 + t)(1+(14s+35s^2+7s^3)t+ (1+14s+98s^2+112s^3+21s^4)t^2 + (14s+35s^2+7s^3)t^3+t^4)+\\
%&(s-1)t{\color{blue}(1+(1+9s+15s^2+s^3)t+(1+14s+36s^2+14s^3+s^4)t^2+(1+9s+15s^2+s^3)t^3+t^4)}.
\end{align*}
For convenience, we list the symmetric decomposition of $A_n(2,t)$ for $n\leqslant 5$:
\begin{align*}
&A_1(2,t)=1,\\
&A_2(2,t)=1+t+t,\\
&A_3(2,t)=(1+9t+t^2)+t(1+t),\\
&A_4(2,t)=(1+31t+31t^2+t^3)+t(1+9t+t^2),\\
&A_5(2,t)=(1+87t+301t^2+87t^3+t^4)+t(1+31t+31t^2+t^3).
%&A_6(2,t)=1+226t+1964t^2+2178t^3+312t^4+2t^5.
\end{align*}
It is well known that $A_n(2,1)$ is the Fubini number~\cite[A000670]{Sloane}, which is the number of ordered partitions of $[n]$.

We now recall an elementary result.
\begin{proposition}[{\cite{Beck2010,Branden18}}]
Let $f(x)$ be a polynomial of degree $n$.
There is a unique decomposition $f(x)= a(x)+xb(x)$, where
\begin{equation}
\label{ax-bx}
a(x)=\frac{f(x)-x^{n+1}f(1/x)}{1-x},~b(x)=\frac{x^nf(1/x)-f(x)}{1-x}.
\end{equation}
\end{proposition}
The ordered pair of polynomials $(a(x),b(x))$ is called the {\it symmetric decomposition} of $f(x)$, since $a(x)$ and $b(x)$ are both symmetric.
Let $f(x)=\sum_{i=0}^nf_ix^i$ be a polynomial with nonnegative coefficients.
Following~\cite[Definition 2.9]{Schepers13}, the polynomial $f(x)$ is {\it alternatingly increasing} if
$$f_0\leqslant f_n\leqslant f_1\leqslant f_{n-1}\leqslant\cdots \leqslant f_{\lrf{{(n+1)}/{2}}}.$$
There has been much recent work devoted to alternatingly increasing polynomials, see~\cite{Branden18,Branden22}.
\begin{lemma}[{\cite[Lemma~2.1]{Beck2019}}]\label{lemma-alt}
Let $(a(x),b(x))$ be the symmetric decomposition of $f(x)$, where $\deg f(x)=\deg a(x)=n$ and $\deg b(x)=n-1$.
Then $f(x)$ is alternatingly increasing if and only if both $a(x)$ and $b(x)$ are unimodal.
\end{lemma}

We can now conclude the main result of this paper.
\begin{theorem}\label{thm20}
For any $n\geqslant 1$, we define the $t$-symmetric polynomial
$$a_n(s,t) = \frac{A_n(s,t) - t^n A_n(s,1/t)}{1-t}.$$
Then the $(\des,~\exc)$-Eulerian polynomial $A_n(s,t)$ has the $t$-symmetric decomposition
\begin{equation}\label{Anst-decom}
A_n(s,t)=a_n(s,t)+(s-1)ta_{n-1}(s,t),
\end{equation}
\end{theorem}
\begin{proof}
Define
$$f(s,t,u)= \sum_{n\geqslant0} A_n(s,t) \frac{u^n}{(1-s)^{n+1}}.$$
According to~\cite[Eq~(1.15),~Eq.~(1.18)]{Foata08}, we have
\begin{equation*}\label{A_n(s,t)gf}
\sum_{n\geqslant 0}A_n(s,t)\frac{u^n}{(1-s)^{n+1}}=\sum_{r\geqslant 0}s^r\frac{1-t}{(1-u)^{r+1}(1-ut)^{-r}-t(1-u)}.
\end{equation*}
So one has
$$f(s,t,u)=\sum_{r\geqslant0}s^r\frac{(1-t)(1-ut)^r}{(1-u)\left((1-u)^r-t(1-ut)^r\right)}.$$

Note that $\deg A_n(s,t)=n-1$ as a polynomial in $t$.
Suppose that $(a_n(s,t),b_n(s,t))$ is the symmetric decomposition of $A_n(s,t)$, i.e.,
$A_n(s,t)=a_n(s,t)+tb_n(s,t)$.
Applying~\eqref{ax-bx}, we have
$$
a_n(s,t) = \frac{A_n(s,t) - t^n A_n(s,1/t)}{1-t}$$
and
$$
b_n(s,t) = \frac{ t^{n-1} A_n(s,1/t) - A_n(s,t)}{1-t}.$$
Thus
\begin{align*}
\sum_{n\geqslant 0} a_n(s,t) \frac{u^n}{(1-s)^{n+1}}
&=\frac{1}{1-t}\left[\sum_{n\geqslant0} A_n(s,t) \frac{u^n}{(1-s)^{n+1}}-\sum_{n\geqslant 0} A_n(s,1/t) \frac{(tu)^n}{(1-s)^{n+1}}\right]\\
&=\frac{1}{1-t}[f(s,t,u)-f(s,1/t,tu)]\\
&=\sum_{r\geqslant0}s^r\frac{(1-ut)^r}{(1-u)\left((1-u)^r-t(1-ut)^r\right)}\\
&\quad -\sum_{r\geqslant0}s^r\frac{-(1-u)^r}{(1-ut)\left(t(1-ut)^r-(1-u)^r\right)}\\
&=\sum_{r\geqslant0}s^r\frac{(1-ut)^{r+1}-(1-u)^{r+1}}{(1-u)(1-ut)\left((1-u)^r-t(1-ut)^r\right)}.
\end{align*}
We now show that $b_n(s,t)=(s-1)a_{n-1}(s,t)$ for $n\geqslant 1$.
Note that
$$b_n(s,t)=\frac{A_n(s,t)-a_n(s,t)}{t},$$
and
$$\sum_{n\geqslant1} (s-1)a_{n-1}(s,t) \frac{u^n}{(1-s)^{n+1}}=-u\sum_{n\geqslant0}a_{n}(s,t) \frac{u^n}{(1-s)^{n+1}}.$$
Set $a_0(s,t)=0$.
Then we obtain
\begin{align*}
&\sum_{n\geqslant1} \left(b_{n}(s,t)-(s-1)a_{n-1}(s,t)\right) \frac{u^n}{(1-s)^{n+1}}\\
&=\frac{1}{t}\sum_{n\geqslant1}\left(A_n(s,t)-(1-ut)a_n(s,t)\right) \frac{u^n}{(1-s)^{n+1}}\\
&=-\frac{1}{t(1-s)}+\frac{1}{t}\sum_{n\geqslant0}\left(A_n(s,t)-(1-ut)a_n(s,t)\right) \frac{u^n}{(1-s)^{n+1}}\\
&=-\frac{1}{t(1-s)}+\frac{1}{t}\sum_{r\geqslant0}s^r\frac{\left((1-ut)-(1-u)t\right)(1-ut)^r}{(1-u)\left((1-u)^r-t(1-ut)^r\right)}\\
&\quad -\frac{1}{t}\sum_{r\geqslant0}s^r\frac{(1-ut)^{r+1}-(1-u)^{r+1}}{(1-u)\left((1-u)^r-t(1-ut)^r\right)}\\
&=-\frac{1}{t(1-s)}+\frac{1}{t}\sum_{r\geqslant0}s^r\frac{(1-u)\left((1-u)^r-t(1-ut)^r\right)}{(1-u)\left((1-u)^r-t(1-ut)^r\right)}\\
&=-\frac{1}{t(1-s)}+\frac{1}{t}\sum_{r\geqslant0}s^r=0,
\end{align*}
which yields the desired result. So we obtain~\eqref{Anst-decom}.
\end{proof}

Theorem~\ref{thm20} may be seen as a key step to understand the polynomials $A_n^{\maj,\des,\exc}(q,p,q^{-1}t)$,
since we can consider the $t$-symmetric decompositions of these polynomials. As illustrations,
\begin{align*}
&A_1^{\maj,\des,\exc}(q,p,q^{-1}t)=1,\\
&A_2^{\maj,\des,\exc}(q,p,q^{-1}t)=1+pt=1+t+(p-1)t,\\
&A_3^{\maj,\des,\exc}(q,p,q^{-1}t)=1+(2p+pq+p^2q^2)t+pt^2\\
&=(1+(1+p+pq+p^2q^2)t+t^2)+(p-1)t(1+t),\\
&A_4^{\maj,\des,\exc}(q,p,q^{-1}t)=1+(3p+2pq+pq^2+2p^2q^2+2p^2q^3+p^2q^4)t+\\
&(3p+pq+p^2q+3p^2q^2+2p^2q^3+p^3q^4)t^2+pt^3\\
&=(1+t)(1+(p+p^2q^2)(2+2q +q^2)t+t^2)+(p-1)t(1+(1+pq+pq^2+p^2q^4)t+t^2).
%&A_5^{\maj,\des,\exc}(q,p,q^{-1}t)=\\
%&1+(4p+3pq+2pq^2+pq^3+3p^2q^2+4p^2q^3+5p^2q^4+2p^2q^5+p^2q^6+p^3q^6)t+\\
%&(5p+3pq+2pq^2+7p^2q^2+11p^2q^3+9p^2q^4+4p^3q^4+6p^3q^5+6p^3q^6+2p^3q^7+p^4q^8)t^2\\
%&(4p+pq+2p^2q+4p^2q^2+5p^2q^3+3p^2q^4+2p^3q^4+3p^3q^5+2p^3q^6)t^3+pt^4.
\end{align*}

Based on empirical evidence, we now propose a conjecture.
%\begin{conjecture}
%Let $s>1$ be a given real number. Then the polynomial $a_n(s,t)$ is $\gamma$-positive, i.e., there exist nonnegative real numbers $\gamma_{n,i}(s)$ such that
%$$a_n(s,t)=\sum_{i=0}^{\lrf{(n-1)/2}}\gamma_{n,i}(s)t^i(1+t)^{n-1-2i}.$$
%Then by Lemma~\ref{lemma-alt}, the polynomial $A_n(s,t)$ is $t$-alternatingly increasing and so it is unimodal with mode in the middle.
%\end{conjecture}
\begin{conjecture}
Let $p>1$ and $q\geqslant 1$ be two given real numbers. Then the two polynomials in the symmetric decomposition of
$A_n^{\maj,\des,\exc}(q,p,q^{-1}t)$ are both $\gamma$-positive. So the polynomial
$A_n^{\maj,\des,\exc}(q,p,q^{-1}t)$ is $t$-alternatingly increasing and it is unimodal with mode in the middle.
\end{conjecture}

Denote the coefficient of $u^n$ in a power series $f(u)$ by $[u^n](f(u))$. In the proof of Theorem~\ref{thm20}, we obtain
\begin{align*}
\frac{a_n(s,t)}{(1-s)^{n+1}}&=[u^n]\left(
\sum_{r\geqslant0}s^r\frac{(1-ut)^{r+1}-(1-u)^{r+1}}{(1-u)(1-ut)\left((1-u)^r-t(1-ut)^r\right)}\right)\\
&=\sum_{r\geqslant0}s^r[u^n]\left(\frac{(1-ut)^{r+1}-(1-u)^{r+1}}{(1-u)(1-ut)\left((1-u)^r-t(1-ut)^r\right)}\right).
\end{align*}
By partial fraction decomposition, we have
\begin{align*}
\frac{a_n(s,t)}{(1-s)^{n+1}}
&=\sum_{r\geqslant0}s^r[u^n]\left(\frac{-1}{t(1-u)}+\frac{-1}{1-ut}+\frac{(1-u)^{r-1}-t^2(1-ut)^{r-1}}{t((1-u)^r-t(1-ut)^r)}\right)\\
&=-\frac{1+t^{n+1}}{t(1-s)}+\frac{1}{t}\sum_{r\geqslant0}s^r[u^n]f(u),
\end{align*}
where $f(u)=\frac{(1-u)^{r-1}-t^2(1-ut)^{r-1}}{(1-u)^r-t(1-ut)^r}$.

Set $[u^n](f(u))=f_n$. By comparing the coefficient of $u^n$ in both sides of $$((1-u)^r-t(1-ut)^r)f(u)=(1-u)^{r-1}-t^2(1-ut)^{r-1},$$
 we obtain
$$\sum_{j=0}^n(-1)^j\binom{r}{j}[j+1]_tf_{n-j}=(-1)^n\binom{r-1}{n}[n+2]_t.$$
with $[n]_t=\frac{1-t^n}{1-t}$.  Define $\alpha_j=\binom{r}{j}[j+1]_t$ and $\beta_j=(-1)^j\binom{r-1}{j}[j+2]_t$ for all $j\geqslant0$. Clearly, $\alpha_j=0$ with $j\geqslant r+1$ and $\beta_j=0$ with $j\geqslant r$. Define ${\bf M}_{n,r}=(m_{ij})$ to be the matrix $(n+1)\times(n+1)$, where $m_{ij}=\alpha_{i-j}$ with $i=0,1,\ldots,n$ and $j=0,1,\ldots,n-1$, and $m_{in}=\beta_i$ for all $i=0,1,\ldots,n$. Clearly, $f_n=\det({\bf M}_{n,r})$.
Therefore, we can state the following result.
\begin{theorem}\label{thT1}
We have
\begin{align*}
a_n(s,t)&=-\frac{(1+t^{n+1})(1-s)^n}{t}+\frac{(1-s)^{n+1}}{t}\sum_{r\geqslant0}s^r\det({\bf M}_{n,r}).
\end{align*}
\end{theorem}

\begin{example}
By Theorem~\ref{thT1}, when $n\leqslant 4$, we see that
\begin{align*}
\det({\bf M}_{0,r})&=[2]_t,~a_0(s,t)=0,~\det({\bf M}_{1,r})=[3]_t+tr,~a_1(s,t)=1,\\
\det({\bf M}_{2,r})&=[4]_t+\frac{3}{2}t(t+1)r+\frac{1}{2}t(t+1)r^2,~a_2(s,t)=1+t,\\
\det({\bf M}_{3,r})&=[5]_t+\frac{1}{6}t(11t^2+14t+11)r+t(t+1)^2r^2+\frac{1}{6}t(t^2+4t+1)r^3,\\
a_3(s,t)&=s^2t+2st+t^2+t+1,\\
\det({\bf M}_{4,r})&=[6]_t+\frac{5}{12}t(5t^3+7t^2+7t+5)r+\frac{5}{24}t(7t^3+17t^2+17t+7)r^2\\
&+\frac{5}{12}t(t+1)(t^2+4t+1)r^3+\frac{1}{24}t(t^3+11t^2+11t+1)r^4,\\
a_4(s,t)&=5s^2t(t+1)+5st(t+1)+t^3+t^2+t+1,\\
\end{align*}
\end{example}
%%%%%%%%%%%%%%%%%%%%%%%%%%%%%%%%%%%%%%%%%%%%%%%%%%%%%%%%%%%%%%%%%%%%%%%%%%%
\section*{Acknowledgements.}
%%%%%%%%%%%%%%%%%%%%%%%%%%%%%%%%%%%%%%%%%%%%%%%%%%%%%%%%%%%%%%%%%%%%%%%%%%%%
%%%%%%%%%%%%%%%%%%%%%%%%%%%%%%%%
%%%%%%%%%%%%%%%%%%%%%%%%%%%%%%%%%%%%%%%%%%%%%%%%%%%%%%%%%%%
%%%%%%%%%%%%%%%%%%%%%%%%%%%5%%%%
%%%%%%%%%%%%%%%%%%%%%%%%%%%%%%%%%%%%%%%%%%%%%%%%%%%%%%%%%%%%%%%%%%%%%%%%%%%
%%%%%%%%%%%%%%%%%%%%%%%%%%%%%%%%%%%%%%%%%%%%%%%%%%%%%%%%%%%%%%%%%%%%%%%%%%
We gratefully acknowledge Bruce Sagan and Qing-Hu Hou for their valuable suggestions and discussions.
The first author was supported by the National Natural Science Foundation of China (Grant number 12071063)
and Taishan Scholars Foundation of Shandong Province (No. tsqn202211146).
%%%%%%%%%%%%%%%%%%%%%%%%%%%%%%%%
%%%%%%%%%%%%%%%%%%%%%%%%%%%%%%%%%%%%%%%%%%%%%%%%%%%%%%%%%%%
%%%%%%%%%%%%%%%%%%%%%%%%%%%5%%%%
%%%%%%%%%%%%%%%%%%%%%%%%%%%%%%%%%%%%%%%%%%%%%%%%%%%%%%%%%%%%%%%%%%%%%%%%%%%
%%%%%%%%%%%%%%%%%%%%%%%%%%%%%%%%%%%%%%%%%%%%%%%%%%%%%%%%%%%%%%%%%%%%%%%%%%
%%%%%%%%%%%%%%%%%%%%%%%%%%%%%%%%%%%%%%%%%%%%%%%%%%%%%%%%%%%%%%%%%%%%%%%%%%
%%%%%%%%%%%%%%%%%%%%%%%%%%%%%%%%%%%%%%%%%%%%%%%%%%%%%%%%%%%%%%%%%%%%%%%%%%
%%%%%%%%%%%%%%%%%%%%%%%%%%%%%%%%%%%%%%%%%%%%%%%%%%%%%%%%%%%%%%%%%%%%%%%%%%


\begin{thebibliography}{37}
\bibitem{Athanasiadis17}
C.A. Athanasiadis, \textit{Gamma-positivity in combinatorics and geometry}, S\'em. Lothar. Combin., \textbf{77} (2018), Article B77i.

\bibitem{Beck2010}
M. Beck, A. Stapledon, \textit{On the log-concavity of Hilbert series of Veronese subrings and Ehrhart series}, Math. Z., \textbf{264} (2010), 195--207.

\bibitem{Beck2019}
M. Beck, K. Jochemko and E. McCullough, \textit{$ h^\ast$-polynomials of zonotopes}, Trans. Amer. Math. Soc., \textbf{371} (2019), 2021--2042.

\bibitem{Branden18}
P. Br\"and\'{e}n, L. Solus, \textit{Symmetric decompositions and real-rootedness}, Int. Math. Res. Not., \textbf{2021} (2021), 7764--7798.

\bibitem{Branden22}
P. Br\"anden, K. Jochemko, \textit{The Eulerian transformation}, Trans. Amer. Math. Soc., \textbf{375} 3 (2022), 1917--1931.

\bibitem{Brenti90}
F. Brenti, \textit{Unimodal polynomials arising from symmetric functions}, Proc. Amer. Math. Soc., \textbf{108} (1990), 1133--1141.


\bibitem{Foata08}
D. Foata and G.-N Han, \textit{Fix-mahonian calculus III; a quadruple distribution}, Monatsh Math., \textbf{154} (2008), 177--197.

\bibitem{Foata70}
D. Foata and M.P. Sch\"utzenberger, \textit{Th\'eorie g\'eometrique des polyn\^{o}mes eul\'eriens}, Lecture Notes in Math., vol. 138, Springer, Berlin, 1970.


\bibitem{He12}
T.X. He, \textit{Eulerian polynomials and B-splines}, J. Comput. Appl. Math., \textbf{236} (2012), 3763-3773.


\bibitem{Hwang20}
H.-K. Hwang, H.-H. Chern, G.-H. Duh, \textit{An asymptotic distribution theory for Eulerian
recurrences with applications}, Adv. in Appl. Math., \textbf{112} (2020), 101960.

\bibitem{Li12}
N. Li, \textit{Ehrhart $h*$-vectors of hypersimplices}, Discr. Comp. Geom., \textbf{48} (2012), 847--878.

\bibitem{Liu07}
L.L. Liu, Y. Wang, \textit{A unified approach to polynomial sequences with only real zeros}, Adv. in Appl. Math., \textbf{38} (2007), 542--560.


\bibitem{Ma19}
S.-M. Ma, J. Ma, Y.-N. Yeh, \textit{$\gamma$-positivity and partial $\gamma$-positivity of descent-type polynomials}, J. Combin. Theory Ser. A, \textbf{167} (2019), 257--293.

%
%\bibitem{Ma2024}
%S.-M. Ma, J. Ma, J. Yeh, Y.-N. Yeh, \textit{Excedance-type polynomials, gamma-positivity
%and alternatingly increasing property}, European J. Combin., \textbf{118} (2024), 103869.

\bibitem{MacMahon15}
P.A. MacMahon, \textit{Combinatory Analysis}, vol. 2, Cambridge University Press, London, 1915-1916. Reprinted by
Chelsea, New York, 1960.

\bibitem{Petersen15}
T.K. Petersen, \textit{Eulerian Numbers}, Birkh\"auser/Springer, New York, 2015.

\bibitem{Schepers13}
J. Schepers and L.V. Langenhoven, \textit{Unimodality questions for integrally closed lattice polytopes}, Ann.
Combin., \textbf{17}(3) (2013), 571--589.

\bibitem{Shareshian10}
J. Shareshian, M.L. Wachs, \textit{Eulerian quasisymmetric functions}, Adv. Math., \textbf{225}(6) (2010), 2921--2966.

\bibitem{Zeng16}
H. Shin and J. Zeng, \textit{Symmetric unimodal expansions of excedances in colored permutations}, European J. Combin., \textbf{52} (2016), 174--196.

\bibitem{Sloane}
N.J.A. Sloane, The On-Line Encyclopedia of Integer Sequences, https://oeis.org.
\end{thebibliography}
\end{document}